\documentclass[12pt]{amsart}
\usepackage{a4wide, amssymb,amsfonts,amsmath,amsthm,amscd,soul}

\newcommand{\Hom}[2]{\mathrm{Hom}(#1,#2)}
\newcommand{\End}[1]{\mathrm{End}(#1)}

\newcommand{\Ann}[1]{\mathrm{Ann}(#1)}
\newcommand{\lann}[2]{\mathrm{l.ann}_{#1}(#2)}
\newcommand{\rann}[2]{\mathrm{r.ann}_{#1}(#2)}

\newcommand{\Soc}[1]{\mathrm{Soc}(#1)}

\newcommand{\ZZ}{\mathbb{Z}}

\theoremstyle{plain}

\newtheorem*{lem}{Lemma}
\newtheorem*{thm}{Theorem}

\newtheorem*{cor}{Corollary}
\newtheorem*{prop}{Proposition}

\title{Endomorphism rings of modules over prime rings}
\author{Mohammad Baziar}
\address{​Department of Mathematics, Yasouj University, Yasouj, Iran, 75914}
\email{ baziar@math.iut.ac.ir}
\author{Christian Lomp}
\address{Department of Mathematics, University of Porto, Porto, Portugal}
\email{clomp@fc.up.pt}
\thanks{ The second author was partially supported by Centro de Matem\'atica da Universidade do Porto (CMUP), financed by FCT (Portugal) through
the programs POCTI (Programa Operacional Ci\^encia, Tecnologia, Inova\c{c}\~ao) and POSI (Programa Operacional Sociedade da Informa\c{c}\~ao), with national and European community structural funds.}

\begin{document}



\begin{abstract}
Endomorphism rings of modules appear as the center of a ring, as the fix ring of ring with group action or as the subring of constants of a derivation.
This note discusses the question whether certain $*$-prime modules have a prime endomorphism ring. Several conditions are presented that guarantee the primness of the endomorphism ring. The contours of a possible example of a $*$-prime module whose endomorphism ring is not prime are traced.
\end{abstract}
\maketitle
\section{Introduction}

Endomorphism rings of modules appear in many ring theoretical situations. For example the center $C(R)$ of a (unital, associative) ring $R$ is isomorphic to the endomorphism ring of $R$ 
seen as a bimodule over itself, i.e. as a left $R\otimes R^{op}$-module. The subring $R^G$ of elements that are left invariant under the action of a group $G$ on $R$ is isomorphic to the 
endomorphism ring of $R$ seen as a left module over the skew group ring $R*G$. The subring $R^\partial$ of constants of a derivation $\partial$ of $R$ is isomorphic to the endomorphism ring of $R$ 
seen as a left module over its differential operator ring $R[x,\partial]$. More generally the subring $R^H$ of elements invariant under the action of a Hopf algebra $H$ acting on $R$ is isomorphic to 
the endomorphism ring of $R$ seen as left module over the smash product $R\# H$. This identifications motivated the use of module theory in the study of Hopf algebra actions in \cite{BorgesLomp11b, Lomp04, Lomp05a, Lomp05b}.

Prime numbers and prime ideals are basic concepts in algebra.  While the idea of a prime ideal is well established, the idea of a prime submodule of a module is not. 
The purely essence of a prime ideal had been destilled already by Birkhoff in the concept of a prime element in a partially ordered groupoid. 
In \cite{bican}, Bican et al. introduced an operation on the lattice of submodules of a module, turning it into a partially ordered groupoid. 
Let $R$ be any (associative, unital) ring and $M$ a left $R$-module. For any submodules $N,L$ we denote 
$$N*L = N\Hom{M}{L} = \sum \{ (N)f \mid f:M\rightarrow L \}.$$
Note that we will write homomorphisms oposite of scalars, i.e. on the right side of an element.
A submodule $P$ is a prime element in $M$ if for any two submodules $N,L$ of $M$
$$ N*L \subseteq P \Rightarrow N\subseteq P \:\:\mbox{ or }\:\: L\subseteq P.$$
Those modules whose zero submodule is a prime element had been termed $*$-prime modules, i.e. $N*L \neq 0$ for all non-zero $N,L\subseteq M$. Of course for $M=R$, the $*$-product equals the product of left ideals and $R$ is a $*$-prime left $R$-module if and only if it is a prime ring. The meaning of the module theoretic prime concept for a ring $R$ with Hopf algebra action $H$ seen as left $R\# H$-module has been studied in \cite{Lomp05b} in connection with 
an open question in this area, due to Miriam Cohen, asking whether $R\# H$ is a semiprime algebra provided $R$ is semiprime and $H$ is semisimple (see \cite{Cohen85}).

The main purpose of this note is to shed new light into the following question which had been left open in \cite{Lomp05b}:

\medskip
{\bf Question:} Is the endomorphism ring of a $*$-prime module a prime ring ?
\medskip

From \cite[Proposition 4.2]{Lomp05b} it is known that the answer is yes, if the $*$-prime module $M$ satisfies a light projectivity condition. 
Although we were unable to answer this question completely we will indicate various sufficient conditions 
for a $*$-prime module to have a prime endomorphism ring which narrows down the class of possible examples that could provide a negative answer. 

Let $M$ be a left $R$-module.  $S=\mathrm{End}_R(M)$ shall always denote the endomorphism ring of $M$. 
Since any $*$-prime module $M$ has a prime annihilator ideal $\Ann{M}$ and since $\mathrm{Hom}_R(M,N)=\mathrm{Hom}_{R/\Ann{M}}(M,N)$ holds for any submodule $N$ of $M$, \emph{\bf we will assume throughout this note that 
$M$ is a faithful left module over a (unital, associative) prime ring $R$}.

\subsection{Retractable modules}
A $*$-prime module $M$ is  \emph{retractable}, i.e.  $\Hom{M}{K}\neq 0$ whenever $0\neq K\subseteq M$.
Note that it is always true that a retractable module with prime endomorphism ring is a $*$-prime module (see \cite[Theorem 4.1]{Lomp05b}) and our question is whether this sufficient condition is also necessary.
The retractability condition (called \emph{quotient like} in \cite{haghany} and \emph{slightly compressible} in \cite{Smith05}) stems from the non-degeneration of the standard Morita context 
$(R,M,M^*,S)$ between a ring $R$ and the endomorphism ring $S$ of a module $M$  via $M^*=\Hom{M}{R}$ (see \cite{Zelmanowitz67}). 
In the case of a group $G$ acting on a ring $R$, the retractability of 
$R$ as $R*G$-module says that every non-zero $G$-stable left ideal contains a non-zero fixed element. 
The Bergman-Isaac theorem \cite{BergmanIsaacs} says that $R$ is retractable as left $R*G$-module 
if $G$ is a finite group acting on a semiprime ring $R$ such that no non-zero element of $R$ has additive $|G|$-torsion. This fact had been used by Fisher and Montgomery in \cite{FisherMontgomery} to prove 
that $R*G$ is semiprime provided $R$ is semiprime and has no $|G|$-torsion, which originally with \cite{CohenRowen} motivated Cohen's question for Hopf algebra actions.

For a locally nilpotent derivation $\partial$ of a ring $R$ it had been shown in \cite[Lemma 3.8]{BorgesLomp11b} that $R$ is açways retractable as $R[x,\partial]$-module.
Rings $R$ that are retractable as $R\otimes R^{op}$-module are those whose non-zero ideals contain non-zero elements like for example in the case of 
semiprime PI-rings (\cite[Theorem 2]{Rowen73}), central Azumaya rings (\cite[26.4]{wisbauer96}) or enveloping algebras of semisimple Lie algebras (\cite[4.2.2]{Dixmier}).
The retractability condition can be expressed by saying that the function from the lattice of left $R$-submodules of the module $M$ to the lattice of left ideals of $S$ defined as 
$N\mapsto \Hom{M}{N}$ for submodules $N$ of $M$ has the property that the only submodule mapped to the zero left ideal of $S$ is the zero submodule.

\subsection{Endoprime modules}
It is known by \cite[1.3]{Lomp05b} that the endomorphism ring of a right R-module $M$ is prime if and only if
$\Hom{M/N}{M}=0$ for all non-zero fully invariant, $M$-generated submodules $N$ of $M$. 
With slightly different notation,  Haghany and Vedadi defined a module $M$ to be \emph{endoprime} if $\Hom{M/K}{M} = 0$ for all non-zero fully invariant submodules $K$ of $M$ (see \cite{haghany}).  
Thus endoprime modules have a prime endomorphism ring. Since for all submodules $K$ of $M$, 
$\Hom{M}{K}\Hom{M/K}{M}=0$, we see that a retractable module $M$ with prime endomorphism ring $S$ is endoprime.
In other words a retractable module has a prime endomorphism ring if and only if it is endoprime. 
Since $*$-prime modules are rectractable, our question can be equivalently reformulated to 

{\bf Question:} Are $*$-prime modules endoprime in the sense of Haghany and Vedadi ?

\subsection{Semi-projective modules}

As mentioned before under a light projectivity condition our question has an affirmative answer.
Recall from \cite{wisbauer96} that a module $M$ is called \emph{semi-projective} if any diagram
\[\begin{CD}  @. M \\
@. @VVgV\\
M @>f>> K @>>> 0\end{CD}\]
with $K\subseteq M$ can be extended by some endomorphism of $M$. In other words, $M$ is semi-projective
if and only if for any endomorphism $f$ of $M$ we have $\Hom{M}{(M)f} = Sf$.

\begin{lem}[{\cite[Proposition 4.2]{Lomp05b}}]\label{semiprojective}
 A semi-projective module is $\star$-prime if and only if it is a retractable module with prime endomorphism ring.
\end{lem}

Let $R$ be a ring and $B\subseteq \mathrm{End}_{\ZZ}(R)$ be a subring of the ring of $\ZZ$-linear endomorphisms of $R$ such that all left multiplications $L_a:R\rightarrow R$ defined by $L_a(x)=ax$ for $a,x\in R$ belong to 
$B$. $R$ becomes naturally a left $B$-module by evaluating of functions. The subring $R^B = \{ (1)f \mid f\in B\}$ can be seen to be a generalized subring of invariants of $R$ with respect to $B$. 
It is not difficult to see, that $R^B$ is isomorphic to $\mathrm{End}_{B}(R)$ (see \cite[Lemma 1.8]{Lomp05a}). This general situation mimics the case of $R$ considered as a bimodule or $R$ considered having
a Hopf algebra $H$ acting on it. To ask that $R$ is a semi-projective as $B$-module, is to say that for each $x\in R^B$ one has $R^B x = (Rx)\cap R^B$.

Considering $R$ as a bimodule, we let $B$ to be the subring of $\mathrm{End}_{\ZZ}(R)$ generated by all left and right multiplications of elements of $R$. The $B$-module structure of $R$ is identical with
the bimodule structure of $R$. Then $R$ is semi-projective as $R\otimes R^{op}$-module if for example all non-zero central elements of $R$ are non-zero divisors in $R$. Because
if $x$ is central and $ax$ is central for some $a\in R$, then for any $b\in R$ one has $(ab-ba)x=abx-bax=axb-axb=0$, i.e. $ab=ba$ and $a$ is central. Thus $Rx\cap C(A) = C(A)x$.
In case  $R$ is  $*$-prime as $R\otimes R^{op}$-module, $0\neq x\in C(R)$ and $I=\Ann{x}=\{a\in R \mid ax=0\}$ is its annihilator, the $*$-product of $I$ and $Rx$ is given by:
$$I*(Rx) = I \mathrm{Hom}_{R\otimes R^{op}}(R, Rx) = I( (Rx)\cap C(R)) \subseteq Ix = 0.$$ Since we supposed that $R$ is $*$-prime and $x\neq 0$, we get  $I=0$.
This shows that no non-zero central element of $R$ is a zero-divisor in $R$. Consequentely we can state the following 

\begin{cor}\label{bimodule}
A ring $R$ is a $*$-prime $R\otimes R^{op}$-module if and only if the center of $R$ is an integral domain and large in $R$.
\end{cor}
Here we say that a subring $R'$ of $R$ is large in $R$ if any non-zero ideal of $R$ contains a non-zero element of $R'$.

Let $G$ be a group acting on $R$. It is known that  $R$ is a projective $R*G$-module if and only if  $G$ is a finite group and $|G|1$ is invertible in $R$. 
Thus in this case $R$ is a $*$-prime $R*G$-module  if and only if $R^G$ is a prime ring.

If $R$ is an algebra over a field $F$ and $\partial$ is a locally nilpotent derivation of $R$  and either $\mathrm{char}(F)=0$ or $\partial^{\mathrm{char}(F)} = 0$, then 
$R$ is self-projective as left $R[x,\partial]$-module by \cite[Proposition 3.10]{BorgesLomp11b}.  Hence in this situation (using also \cite[Lemma 3.8]{BorgesLomp11b}) $R$ is a $*$-prime left $R[x,\partial]$-module if and only if $R^\partial$ is a
prime ring.

\section{Prime endomorphism rings}

The purpose of this section is to gather conditions for a $*$-prime module to have a prime endomorphism ring.
Denote by $\lann{S}{I}$ (resp. by $\rann{S}{I}$) the left (resp. right) annihilator in $S$ of an ideal $I$.

\begin{thm}\label{prime-semiprime}
The following statements are equivalent for a $*$-prime module $M$ with endomorphism ring $S$:
\begin{enumerate}
 \item[(a)] $S$ is prime.
 \item[(b)] $S$ is semiprime. 
 \item[(c)] $\lann{S}{I}\subseteq \rann{S}{I}$ holds for any ideal $I$ of $S$.
 \item[(d)] $gSf=0 \Rightarrow fSg=0$ for all $f,g \in S$.
\end{enumerate}
\end{thm}

\begin{proof}
 $(a)\Rightarrow (b) \Rightarrow (c)$ is trivial since the left and right annihilator of an ideal coincide in a semiprime ring.
$(c)\Rightarrow (a)$ Suppose that $IJ=0$ for two ideals $I,J$ of $S$. 
Then $M\Hom{M}{MI}J\subseteq MIJ=0$ implies $\Hom{M}{MI}J=0$.
By (c) $J\Hom{M}{MI}=0$.  Hence $ (MJ)*(MI) = MfS\Hom{M}{MI}=0$ and since $M$ is $*$-prime, we have $MI=0$ or $MJ=0$, i.e. $I=0$ or $J=0$. Thus  $S$ is prime.

Condition $(d)$ is equivalent to saying that 
$$\lann{S}{SfS}=\lann{S}{Sf} \subseteq \rann{S}{fS}=\rann{S}{SfS}$$ for all $f\in S$, which is a consequence of $(c)$.
On the other hand, assuming $(d)$ condition $(c)$ follows since for any non-zero ideal $I$ we have 
$\lann{S}{I}=\bigcap_{f\in I} \lann{S}{SfS}$  and the analogous statement for $\rann{S}{I}$.
\end{proof}

Note that $(c)\Rightarrow (a)$ needed only the primness condition for fully invariant submodules. These modules had been investigated by R.Wisbauer and I. Wijayanti and termed \emph{fully prime} modules.

\subsection{}

We deduce two corollaries from the last theorem:

\begin{cor}
 Let $M$ be a left $R$-module with endomorphism ring $S$. Then $S$ is prime and $M$ is rectractable if and only if $M$ is $*$-prime and 
 $gSf=0$ implies $fSg=0$ for all $f,g \in S$.
\end{cor}

As a particular case we recover the characterisation of $R$ being  $*$-prime as bimodule (see \ref{bimodule}):
\begin{cor}
 Let $M$ be a left $R$-module with commutative endomorphism ring $S$.
Then $M$ is $*$-prime if and only if $M$ is retractable and $S$ is an integral domain.
\end{cor}

Since semprime PI-rings or central Azumaya rings have large center, we see that any such ring is a $*$-prime bimodule if and only if its center is a domain.

\subsection{}

The next result generalizes the fact that semi-projective $*$-prime modules have prime endomorphism.

\begin{prop}
 Assume that for any non-zero ideal $J$ of $S$ which is essential as left and right ideal 
 there exists a non-zero submodule $N$ of $M$ such that $\Hom{M}{N}\subseteq J$. Then $S$ is prime if $M$ is $*$-prime.
\end{prop}
\begin{proof}
 Let $I^2=0$ for an ideal in $S$. 
 Then $J=\rann{S}{I}\cap\lann{S}{I}$ is a non-zero ideal of $S$ which is essential on both sides. 
 By assumption $\Hom{M}{N}\subseteq J$ for some non-zero submodule $N$ of $M$. 
 Thus $MI*N = MI\Hom{M}{N} \subseteq MIJ=0$ and as $M$ is $*$-prime and $N$ non-zero we have $I=0$, i.e. $S$ is semiprime and by Theorem \ref{prime-semiprime} $S$ is prime.
\end{proof}

\subsection{}

A left $R$-module $M$ is called \emph{torsionless} if it is cogenerated by $R$. 
A result by Amitsur says that any faithful torsionless module over a prime ring has a prime endomorphism ring (see \cite[Corollary 2.8]{Amitsur71}). 
The following Proposition gives sufficient conditions for a $*$-prime module $M$ to be torsionless.

\begin{prop}\label{amitsor}
 Let $M$ be a faithful left $R$-module over a prime ring $R$.
 In any of the following cases $M$ is torsionless and hence has a prime endomorphism ring.
\begin{enumerate}
 \item $M$ is a $*$-prime module and is not a singular left $R$-module.
 \item $M$ is a $*$-prime module and $R$ is a left duo ring, i.e. any left ideal is twosided.
 \item $M$ is non-singular and is cogenerated by all of its essential submodules.
\end{enumerate}
\end{prop}
\begin{proof}
Note that any non-zero submodule $N$ of $M$ that is not singular contains a submodule which is isomorphic to a non-zero left ideal of $R$. To see this let $0\neq x \in N$ be an element whose annihilator 
$A=\lann{R}{x}$ is not essential in $R$. Let $B$ be a complement of $A$, i.e. a left ideal of $R$ which is maximal with respect to $A\cap B=0$. Then $I=A\oplus B$ is an essential left ideal of $R$
and $Ix\neq 0$ since $B$ is non-zero. As $B\simeq Ix$, we see that $B$ is isomorphic to a submodule of $M$.

(1) As explained above, if $M$ is not singular, then there exists a non-zero left ideal $B$ of $R$ which is isomorphic to a submodule of $M$.
Since $M$ is cogenerated by any of its non-zero submodules, it is cogenerated by $B$ and hence by $R$ as $B\subseteq R$. Thus $M$ is torsionless

(2) Since $M$ is a (faithful) prime module, every submodule is also faithful. 
By hypothesis $I=\lann{R}{m}$ is two sided for any element $m$ of $R$ and hence $0=\Ann{Rm}=\Ann{R/I}=I$, 
i.e. $M$ is not singular and the result follows from (1).

(3) Let $M$ be any non-zero nonsingular module that cogenerated by every essential submodule of itself. By Zorn's Lemma
there exist a maximal direct sum $\bigoplus_I C_i$ of cyclic modules $C_i=Rm_i$ non of which is singular. 
Let $A_i=\lann{R}{m_i}$ for each $i\in I$. Since $A_i$ is not essential in $R$, there exists a non-zero complement $B_i$ of $A_i$ in $R$ 
such that $K_i=A_i\oplus B_i$ is an essential left ideal of $R$. Let $a$ be any element in $R$ such that $a\not\in A_i$. 
Then there exists an essential left ideal $E$ of $R$ such that $Ea = Ra \cap K_i$. Because $M$ is nonsingular, we have 
$ 0\neq Eam_i\subseteq K_im_i \cap Ram_i$. Thus $K_im_i$ is an essential submodule of $C_i$ and moreover $K_i m_i \simeq B_i$.
Hence $N=\bigoplus_{i\in I} K_im_i$ is essential in $M$ and by hypothesis cogenerates $M$. Since $N\simeq \bigoplus_{i\in I} B_i \subseteq R^{(I)}$, $M$ is torsionless.
 \end{proof}

\subsection{}
The Wisbauer category of a module $M$ is the full subcategory of $R$-Mod consisting of submodules of quotients of direct sums of copies of $M$. For $M=R$, we have $\sigma[R]=R$-Mod. A module $N\in \sigma[M]$ is called $M$-singular if there are modules $K,L\in\sigma[M]$  with $K$ being an essential submodule of $L$ and $N\simeq L/K$. 
For $M=R$, $R$-singular modules are called singular.
A module $M$ is called {\it polyform} or {\it non-$M$-singular} if it does not contain any $M$-singular submodule or equivalently if $\Hom{L/K}{M}=0$ for 
all essential submodules $K\subseteq L \subseteq M$ (see \cite{wisbauer96}).

\begin{prop}\label{polyform}
 The endomorphism ring  of a $*$-prime polyform module is a prime ring.
\end{prop}
\begin{proof}
Recall our general hypothesis that $M$ is a faithful left module over a prime ring $R$. Let $I^2=0$ for some ideal $I$ of $S$. Then $MI$ is fully invariant.
Note that any fully invariant submodule $N$ of $M$ is essential as  $M$ is $*$-prime, because 
 for any non-zero submodule $L$ of $M$  we have that  $0\neq N*L=N\Hom{M}{L}\subseteq N\cap L$. Thus $MI$ is essential in $M$.
 Denote by $\pi:M\rightarrow M/MI$  the canonical projection, then $\pi I\subseteq \Hom{M/MI}{M}=0$ as $M$ is polyform.
 Thus $I=0$ and $S$ is semiprime. By Theorem \ref{prime-semiprime} $S$ is prime.
\end{proof}

\section{Simple submodules in weakly compressible modules}
The purpose of this last section is to see what can be said about the endomorphism ring of a $*$-prime module with non-zero socle.
 It is also clear that if a $*$-prime module contains a simple submodule $S$, then any simple submodule of $M$ must be isomorphic to $S$. 
Moreover since $\Soc{M}$, the socle of $M$, is fully invariant, we have for any submodule $L$ of $M$: $(\Soc{M}*L=N\Hom{M}{L}\subseteq \Soc{M}\cap L$. Thus a $*$-prime module $M$ has
either zero socle or has an essential and homogeneous semisimple socle, i.e. isomorphic to a direct sum of copies of a simple module.

A submodule $N$ is called semiprime if for any $K \subseteq M: K*K\subseteq N \Rightarrow K\subseteq N$. A module whose zero submodule is semiprime is called 
{\it weakly compressible} by Zelmanowitz (see \cite{Zelmanowitz95}). Obviously $*$-prime modules are weakly compressible.

\begin{lem}
Any simple submodule of a weakly compressible module $M$ is a direct summand.
\end{lem}

\begin{proof}
Let $K$ be a simple submodule of $M$, then $0\neq K*K=K\Hom{M}{K}$ implies the existence of  $f:M\rightarrow K$ such that $f(K)$ is non-zero, i.e. $f(K)=K$ as $K$ is simple. By Schur's Lemma $\End{K}$ is a division ring and hence there exists an inverse $g\in\End{K}$ of $f$ restricted to $K$, i.e. $gf=id_K$.  Considering $g$ as a map from $K$ to $M$ we showed that $f$ splits, i.e. $K$ is a direct summand of $M$.
\end{proof}

\subsection{}
Since by the last Lemma, simple modules of a $*$-prime module are direct summands, we have the following

\begin{cor}\label{homsem}
Any weakly compressible module with DCC or ACC on direct summands and non-zero socle is homogeneous semisimple.
\end{cor}

\subsection{}
Recall that a ring $R$ is said to be left quotient finite dimensional (qfd) if every cyclic left $R$-module has finite Goldie dimension. 
Any left noetherian or more general any ring with Krull dimension is qfd.

\begin{thm}\label{semilocal_qfd}
Let $R$ be a semilocal or a left qfd ring, then any $*$-prime module with non-zero socle has a prime endomorphism ring.
\end{thm}

\begin{proof}
If $M$ is a $*$-prime module with a non-zero socle, then $\Soc{M}$ is essential and homogeneous semisimple. 
Any cyclic $C$ submodule of $M$ is also a $*$-prime module with non-zero essential socle and by assumption has finite Goldie dimension (in case $R$ is qfd) or 
finite dual Goldie dimension (in case $R$ is semilocal). In either case $C$ has ACC on direct summands and by Corollary \ref{homsem} $C$ is homogeneous simple. 
Thus $M=\Soc{M}\simeq E^{(\Lambda)}$ is homogeneous semisimple and $\End{M}\simeq \End{E^{(\Lambda)}}$ is a prime ring.
\end{proof}

\subsection{}

Recall that a ring $R$ has left Krull dimension $0$ if it is left artinian and left Krull dimension $1$ if every proper cyclic left $R$-module $M\neq R$ is artinian.

\begin{prop}
 Any $*$-prime left module over a ring with left Krull dimension less or equal to $1$ has a prime endomorphism ring.
\end{prop}

\begin{proof}
 Let $R$ be a ring with left Krull dimension $\leq 1$ and let $M$ be a $*$-prime left $R$-module. 
 If $M$ is not singular, then it has a prime endomorphism ring by Proposition \ref{amitsor}. 
 Suppose that $M$ is singular and let $C$ be a non-zero cyclic submodule of $M$, then $C$ is also singular and hence proper, i.e. $C\simeq R/I$ with $I\neq 0$. 
 By hypothesis $R$ has Krull dimension $\leq 1$ and thus $C$ is a artinian. This shows that $M$ has a non-zero socle. By \ref{semilocal_qfd} $M$ has a prime endomorphism.
\end{proof}

This implies that for instance any $*$-prime module over the first Weyl algebra $A_1$ has a prime endomorphism ring.

\section{Conclusion}
Let $C(R)$ denote the center of $R$. 
Faithful *-prime modules $M$ that are not singular have prime endomorphism ring by Proposition \ref{amitsor}. This applies in particular to the following case:
\begin{enumerate}
    \item if $M$ has a non-zero submodule which is finitely generated over $C(R)$ or
    \item if $R$ has a non-zero left ideal which is finitely generated over $C(R)$ or
    \item if $R$ has non-zero left socle
\end{enumerate}

In case $(1)$, if $N=C(R)x_1+\cdots+C(R)x_n$, then $$0=\Ann{M}=\Ann{N}=\Ann{x_1}\cap\cdots \cap \Ann{x_n}.$$
Thus not all of the elements $x_i$ can be singular and $M$ is not a singular module. Case $(2)$ reduces to the first case, because if $I$ is a non-zero left ideal of $R$ which is finitely generated 
over $C(R)$, then since $M$ is faithful, there must exists a non-zero element $m\in M$ with $N=Im\neq 0$. But then $N$ is a non-zero submodule of $M$ which is finitely generated over $C(R)$ and $(1)$ applies.

In case $(3)$ we also see that due to $0=\Ann{M}=\bigcap_{x\in M} \Ann{x}$  not all the annihilators $\Ann{x}$ can be essential left ideals, since otherwise the left socle would be contained in 
$\Ann{M}$ and would be zero. Hence $M$ is not a singular module.

From the preceeding we can conclude that if there exists a $*$-prime faithful left $R$-module $M$ whose endomorphism ring is not prime, then

\begin{itemize}
 \item $R$ is not a left duo ring;
 \item $R$ has zero left socle 
 \item $R$ does not contain any non-zero left ideal which is finitely generated over $C(R)$;
 \item the Krull dimension of $R$ is greater than $1$;
\item $\End{M}$ is not commutative;
\item $M$ is a singular left $R$-module which is neither torsionless nor semi-projective;
\item $M$ is not polyform, i.e. $M$ is cogenerated by some $M$-singular submodule;
\item no non-zero submodule of $M$ is finitely generated over the centre $C(R)$ of $R$;
\item if $M$ has non-zero socle, then $R$ cannot be semilocal nor can $R$ have Krull dimension.
\end{itemize}

\bibliographystyle{plain}
\bibliography{BaziarLomp}


\end{document}